\documentclass{daj}

\dajAUTHORdetails{%
  title = {Sharp Density Bounds on the Finite Field Kakeya Problem}, %% please capitalize all significant words
  author = {Boris Bukh and Ting-Wei Chao},
  plaintextauthor = {Boris Bukh, Ting-Wei Chao},
  plaintexttitle = {Sharp Density Bounds on the Finite Field Kakeya Problem}, %%  title without math or LaTeX
  keywords = {Kakeya problem, lines, directions.},
}   %%% END \dajAUTHORdetails

\dajEDITORdetails{%
   year={2021},
   number={26},
   received={3 August 2021},   % received date: example: 7 January 2017
   published={14 December 2021},  % published date
   doi={10.19086/da.30707},       % XXX = number of paper, e.g. da006 for paper#6
}   %%% END \dajEDITORdetails

\usepackage{amsfonts, amsmath, amssymb}

\usepackage{bookmark}
\bookmarksetup{open,numbered,depth=5} % paragraphs go into outlines

\usepackage{amsthm}

\usepackage{tikz}
\usetikzlibrary{calc,patterns}
\usepackage{mathtools}

\usepackage[nameinlink]{cleveref}

\newtheorem{theorem}{Theorem}
\newtheorem{lemma}[theorem]{Lemma}

\makeatletter
\newcommand{\neutralize}[1]{\expandafter\let\csname c@#1\endcsname\count@}
\makeatother

\newenvironment{theoremprime}[1]
  {\renewcommand{\thetheorem}{\ref*{#1}$'$}%
   \neutralize{theorem}\phantomsection
   \begin{theorem}}
  {\end{theorem}}

  \newtheorem{proposition}[theorem]{Proposition}

\crefname{claim}{Claim}{Claims}

\theoremstyle{definition}
\newtheorem{definition}[theorem]{Definition}

\newcommand*{\eqdef}{\stackrel{\mbox{\normalfont\tiny def}}{=}}   % definition by equality                                      
\newcommand*{\Znneg}{\mathbb{Z}_{\geq 0}}                          % Nonnegative integers
\newcommand*{\Rnneg}{\mathbb{R}_{\geq 0}}                          % Nonnegative real numbers
\newcommand*{\E}{\mathbb{E}}                                     % Expectation
\newcommand*{\Fq}{\mathbb{F}_q}                                  % Finite field
\DeclarePairedDelimiter\abs{\lvert}{\rvert}                     % Absolute values, cardinality
\newcommand*{\polytope}{\mskip1mu\tikz{\draw (0,0)--(0,1.5ex)-- (0.7ex,1.5ex) -- (1.5ex,0.7ex)--(1.5ex,0)--cycle;}\mskip1mu}
\newcommand*{\squarepolytope}{\mskip1mu\tikz{\draw (0,0)--(0,1.5ex)--(1.5ex,1.5ex)--(1.5ex,0)--cycle;}\mskip1mu}
\newcommand*{\simplex}{\mskip1mu\tikz{\draw (0,0)--(0.75ex,1.5ex)--(1.5ex,0)--cycle;}\mskip1mu}
\newcommand*{\simplexone}{\simplex_1}
\newcommand*{\simplextwo}{\simplex_2}

\DeclareMathOperator{\mult}{mult}                              % multiplicity
\DeclareMathOperator{\vol}{vol}                                % volume
\DeclareMathOperator{\codim}{codim}                            % Codimension

\newcommand*{\arXiv}[1]{\href{http://arxiv.org/pdf/#1}{arXiv:#1}}

\begin{document}

\begin{frontmatter}[classification=text]
%% EDITOR: this will force the keywords to appear right after the Abstract.
%%   If the abstract is too long and would force the keywords off the
%%   front page, please comment out % [classification=text] above
%%   This way the keywords will be floated on the bottom of the first page
%%   even though the Abstract spills over to the next page.

%%% AUTHOR: Title goes here.  This line is optional.  You must use it
%%   if title has footnote attached or requires nontrivial typesetting,
%%   e.g., inclusion of linebreaks to force nice layout.
%\title{Short Proof of R\"odl's $n^{\log\log n}$ Bound\footnote{This is a footnote to the title}} %% please capitalize all significant words

%%% AUTHOR:
%%% List all authors. If you wish, place grant acknowledgements in \thanks.
%%% In brackets include a short tag for each author.
\newcommand*\samethanks[1][\value{footnote}]{\footnotemark[#1]}
%\author[bbukh]{Boris Bukh}
%\author[tchao]{Ting-Wei Chao}
\author[pgom]{Boris Bukh\thanks{Supported in part by U.S.\ taxpayers through NSF CAREER grant DMS-1555149}}
\author[joha]{Ting-Wei Chao\samethanks}

%%% AUTHOR: Abstract goes here
\begin{abstract}
A \emph{Kakeya} set in $\Fq^n$ is a set containing a line in every direction. We
show that every Kakeya set in $\Fq^n$ has density at least $1/2^{n-1}$, matching
the construction by Dvir, Kopparty, Saraf and Sudan.
\end{abstract}
\end{frontmatter}

%%% AUTHOR: body of paper starts here
\section{Introduction}
\paragraph{Kakeya sets in finite fields.}
Let $\Fq$ be the finite field of size~$q$. A subset $K\subseteq \Fq^n$ is \emph{Kakeya} if, for every direction $b\in \Fq^n\setminus\{0\}$, 
  it contains a line of the form $\{a+bt : t\in \Fq\}$. In a 2008 breakthrough, Dvir \cite{dvir} proved that every Kakeya set in~$\Fq^n$
  is of size at least~$\frac{1}{n!}q^n$. This was the first lower bound of the order $\Omega_n(q^n)$.
  The constant $1/n!$ was improved to $c^{-n}$ for some $c\approx 2.5$ by Saraf and Sudan~\cite{saraf_sudan},
  who also presented a construction (due to Dvir, with modifications by themselves and Kopparty)
  of a Kakeya set of size $2^{-n+1}q^{n}+O_n(q^{n-1})$.
  The best general lower bound on the size of Kakeya sets, which was found shortly thereafter by Dvir, Kopparty, Saraf, and Sudan \cite{dkss},
  is
  \begin{equation}\label{eq:dkss}
    \abs{K}\geq (2-1/q)^{-n} q^n\qquad\text{for every Kakeya set }K\subseteq\Fq^n.
  \end{equation}
  Motivated by the applications to randomness extractors, they also proved an extension of this
  bound to sets obtained by replacing the notion of a `line' in the definition of the Kakeya
  set by that of an `algebraic curve of bounded degree'.

  The only result breaching the factor-of-two gap between the construction in \cite{saraf_sudan} and the lower bound in \cite{dkss}
  is that of Lund, Saraf and Wolf \cite{lund_saraf_wolf} who proved that $\abs{K}\geq 0.2107q^3$ for Kakeya sets in~$\Fq^3$.

\paragraph{Our result.}
  We improve the lower bound \eqref{eq:dkss} by a factor of $2-1/q$, thereby
  closing the factor-of-two gap in all dimensions.   
  \begin{theorem}\label{thm:high}
  The size of every Kakeya set $K\subseteq\Fq^n$ satisfies
  \[
    \abs{K}\geq (2-1/q)^{-(n-1)}q^n.
  \]
  \end{theorem}

\paragraph{Paper organization.}
We begin by presenting a proof of a slightly weaker
bound in dimension $n=3$ in \Cref{sec:three}.
Though this proof does \emph{not} seem to generalize to the $n>3$,
it illustrates one of the ideas used in the general case.

After presenting the proof of \Cref{thm:high} in \Cref{sec:advanced},
we finish with a brief discussion of the remaining gap between our bound
and the known constructions in \Cref{sec:problem}.

\section{Simple argument in dimension \texorpdfstring{$3$}{3}}\label{sec:three}
In this section we shall prove the following.
\begin{theorem}\label{thm:three}
The size of every Kakeya set $K\subseteq \Fq^3$ satisfies $\abs{K}\geq \frac{1}{4}(q^3+q^2)$.
\end{theorem}

Let
\[
  A\eqdef \{(\alpha_1,\alpha_2,\alpha_3)\in\Znneg^3 : \alpha_1+\alpha_2+\alpha_3< 2q,\text{ and } \alpha_1, \alpha_2<q\},
\]
and consider the vector space of polynomials in $\Fq[x_1,x_2,x_3]$ whose monomials are indexed by $A$,
\[
  V\eqdef \Bigl\{\sum_{\alpha\in A} c_{\alpha} x^\alpha : c_\alpha\in \Fq\Bigr\},
\]
where $x^{\alpha}\eqdef x_1^{\alpha_1}x_2^{\alpha_2}x_3^{\alpha_3}$.
We say that a polynomial $P\in \Fq[x_1,x_2,x_3]$ vanishes at $p\in\Fq^3$ to order $2$ if
$P(p)=0$ and $\nabla P(p)=0$.

\begin{lemma}\label{lemma:3dim}
  Let $K$ be a Kakeya set in $\Fq^3$.
  If a polynomial $P\in V$ vanishes to order $2$ at every point of $K$, then $P$ is the zero polynomial.
\end{lemma}

Before proving the lemma, let us see how to derive \Cref{thm:three} from it. For any $p\in \Fq^3$,
the polynomials vanishing at $p$ to order $2$ form a subspace of codimension $4$ in~$V$. So,
the polynomials vanishing to order $2$ at all points of $K$ form a subspace of codimension at most $4\abs{K}$
in $V$. According to the lemma, the latter subspace is trivial, and so
\[
  4\abs{K}\geq \dim V=\abs{A}=\sum_{\alpha_1,\alpha_2=0}^{q-1}\bigl(2q-\alpha_1-\alpha_2\bigr)=q^3+q^2,
\]
as desired.

\begin{proof}[Proof of \Cref{lemma:3dim}]
  Assume that, on the contrary, $P\neq 0$, and write it as $P=P_0+P_1+\cdots+P_d$, where $P_k$ is the homogeneous component of degree $k$ and $d\eqdef \deg P$.
  Given a line $\ell=\{a +b t:t\in \mathbb{F}_q\}$ inside $K$, define the univariate polynomial $P_\ell (t)\eqdef P(a +b t)$.
  %For each direction $b\in\Fq^3\setminus\{0\}$, there is a line $\ell$ of the form$ \{a +b t:t\in \mathbb{F}_q\}$ inside $K$, and we define the univariate polynomial $P_\ell (t)\eqdef P(a +b t)$.
  Since $P$ vanishes at every point of $\ell$ to order $2$, the polynomial $P_{\ell}$ vanishes at all points of~$\Fq$ to order~$2$.
  Because $\deg P_\ell\leq \deg P<2q$, this implies that $P_\ell$ is the zero polynomial.
  Since the coefficient of $t^d$ in $P_{\ell}$ is $P_d(b)$, it follows $P_d(b)=0$.

  Since $K$ is Kakeya, this means that $P_d(b)=0$ for every $b\neq 0$. In particular, $P_d(b_1,b_2,1)=0$ for all $b_1,b_2\in\Fq$.
  The polynomial $Q(x_1,x_2)\eqdef P_d(x_1,x_2,1)$ is of degree less than $q$ in each of~$x_1$ and~$x_2$.
  Write $Q$ as $Q(x_1,x_2)=\sum_{i=0}^{q-1} Q_i(x_2)x_1^i$ where $\deg Q_i<q$.
  Since $Q$ vanishes identically on $\Fq^2$, the polynomial $Q(x_1,c)\in \Fq[x_1]$ vanishes on $\Fq$, for every choice of $c\in \Fq$.
  As this polynomial is of degree less than $q$, this means that $Q_i(c)=0$ for every $i$ and every~$c$. Since
  $Q_i$'s are themselves of degree less than $q$, it follows that they are zero as polynomials, and so is~$Q$.
  Because $P_d$ is homogeneous and $Q(x_1,x_2)=P_d(x_1,x_2,1)$, this implies that $P_d$ is zero as well, contrary to $\deg P=d$.
\end{proof}

\section{Proof of \texorpdfstring{\Cref{thm:high}}{Theorem 1}}\label{sec:advanced}
\paragraph{Stronger result.}
The result we prove is in fact slightly stronger than \Cref{thm:high}.
We call any line of the form $\{a+bt: t\in\Fq\}$ with $b=(b_1,\dotsc,b_{n-1},1)$ \emph{non-horizontal}.
A set $K\subseteq \Fq^n$ is \emph{almost Kakeya} if it contains a line in every non-horizontal direction.
\begin{theoremprime}{thm:high}\label{thm:highstrong}
The size of every almost Kakeya set $K\subseteq \Fq^n$ satisfies $\abs{K}\geq (2-1/q)^{-(n-1)}q^n$.
\end{theoremprime}

\paragraph{Proof outline.} Like the proof of the case $n=3$, we shall use
polynomials built out of the monomials $x^{\alpha}$ in which the exponent of $x_n$
is less constrained than the exponents of $x_1,\dotsc,x_{n-1}$. As is common
in the other proofs in the area we shall use polynomials vanishing to
high order at points of $K$, and not merely to order $2$.
However, this is not enough to obtain the factor-of-two improvement we seek,
and to bridge the gap we use the idea of Ruixiang Zhang.
Like in his work on multijoints \cite{zhang}, our vanishing
conditions at a point $p\in K$ depend on the lines through~$p$.
The actual conditions are quite different from those in \cite{zhang} though. A similar general idea was used in \cite{yu_zhao} to obtain the sharp constant in the joints problem.

\paragraph{Hasse derivatives and high-order vanishing along lines.}
Over finite fields, all derivatives of order higher than the field's characteristic vanish.
The standard workaround is to employ Hasse derivatives, whose definition and properties we recall.
For more extensive discussion of Hasse derivatives, including proofs of their properties, see \cite[Section 2]{dkss}.

\begin{definition}[Hasse derivatives]
The \emph{Hasse derivatives} of a polynomial $P(x)\in\Fq [x_1,\dotsc,x_n]$ are the polynomials $P^{(i)}(x),i\in \Znneg^{n}$ such that
\[P(x+y)=\sum_{i}P^{(i)}(x)y^i.\]
We say that $P^{(i)}$ is the $i$-th Hasse derivative of~$P$.
\end{definition}
\begin{definition}[Multiplicities]
The \emph{multiplicity} $\mult (P,p)$ of a polynomial $P(x)\in\Fq [x_1,\dotsc,x_n]$ at point $p\in\Fq^n$ is the largest integer $m$ such that $P^{(i)}(p)=0$ for all $i=(i_1,\ldots,i_n)$ such that $\abs{i}\eqdef i_1+\dotsb+i_n<m$. We say that the polynomial $P$ vanishes to order $\mult (P,p)$ at~$p$.

Given a line $\ell=\{a+bt:t\in\Fq\}$ and a polynomial $P\in \Fq[x_1,\dotsc,x_n]$, we say
that $P$ vanishes to \emph{order $m$ at the point $p=a+bt_0$ along $\ell$}
if the univariate polynomial $P(a+bt)$ vanishes to order $m$ at $t=t_0$.
We write $\mult_{\ell}(P,p)$ for the order of vanishing
of $P$ along $\ell$ at the point~$p$.
\end{definition}
Note that, by the third part of the following proposition, $\mult_{\ell}(P,p)$ does not depend on the parameterization of the line $\ell$.

\begin{proposition}[Properties of Hasse derivatives]\label{prop:hasse}\
\begin{itemize}
\item The map $P\mapsto P^{(i)}$ is a linear operator on $\Fq [x_1,\dotsc,x_n]$, for every $i\in \Znneg^{n}$.
\item If we write the polynomial $P$ as $P(x)=\sum_{\alpha\in\mathbb{Z}_{\geq 0}^n}c_\alpha x^\alpha$, then the Hasse derivatives of $P$ are given by
\[P^{(i)}(x)=\sum_{\alpha}\binom{\alpha}{i}c_{\alpha} x^{\alpha-i},\]
where $\binom{\alpha}{i}\eqdef\prod_{k=1}^n\binom{\alpha_k}{i_k}$.
\item If $A$ is an invertible affine transformation of $\Fq^n$, then 
\[\mult(P\circ A,p)=\mult(P,Ap).\]
\end{itemize}
\end{proposition}
\begin{lemma}[Generalized Schwartz--Zippel lemma]\label{lemma:SZ}
If a non-zero polynomial $P(x)\in\Fq [x_1,\dotsc,x_n]$ vanishes to order at least $r$ at every point of $\Fq^n$,
then $\deg P\geq rq$.
\end{lemma}

\paragraph{Proof of \texorpdfstring{\Cref{thm:highstrong}}{Theorem 1'}.}
For a non-horizontal direction $b$, there might be several lines in direction $b$ contained in~$K$. We select one such
line for each $b$, and let $L$ be the resulting set of lines. Note that $\abs{L}=q^{n-1}$.
From now on we shall work exclusively with the lines in $L$, ignoring the other lines that might be contained in~$K$.
For each point $p\in K$, let $L_p\eqdef \{\ell \in L : \ell \ni p\}$ be the lines containing~$p$.

For each line $\ell\in L$, and each point $p$ of $\ell$, we shall
impose vanishing conditions at $p$ described in the following definition.

\begin{definition}
Let $\ell$ be a non-horizontal line, and $p\in \ell$.
We say that a polynomial \emph{$P$ vanishes to order $(r,r')$ at $p$ along $\ell$} if $\mult_{\ell}(P^{(i,0)},p)\geq r'-|i|/q$ for all $i\in\Znneg^{n-1}$ such that $|i|< r$.
\end{definition}
The definition is motivated by the proof in \cite{dkss}, which implicitly uses that $P$ vanishes to order
$(r,r)$ at $p$ along $\ell$ whenever $P$ vanishes to order $(2-1/q)r$ at~$p$.

Throughout this proof, the integer $r'$ will be equal to $r'=(2-1/q)r$, where $r$ is some multiple of~$q$,
which eventually will tend to infinity.

Let
\[
  A=\{\alpha \in\Znneg^n : \abs{\alpha}<(2-1/q)rq,\text{ and } \alpha_1+\dotsb+\alpha_{n-1}<rq\},
\]
and consider the vector space of polynomials with the monomials indexed by $A$,
\[
  V\eqdef \Bigl\{\sum_{\alpha\in A} c_{\alpha} x^\alpha : c_\alpha\in \Fq\Bigr\}.
\]

\begin{lemma}\label{lemma:ZeroPoly}
If $P\in V$ vanishes to order $\bigl(r,(2-1/q)r\bigr)$ at $p$ along $\ell$, for every $p\in K$ and every~$\ell\in L_p$, then $P$ is the zero polynomial.
\end{lemma}
\begin{proof}
  Assume that, on the contrary, $P$ is not the zero polynomial.
  Let $\ell\in L$ be an arbitrary line, and $\abs{i}<r$. By the lemma's assumption, $P^{(i,0)}$ vanishes to order $(2-1/q)r-\lfloor |i|/q \rfloor$ at $p$ along $\ell$.
  Since $\deg P^{(i,0)}< (2-1/q)rq-|i|\leq q\big((2-1/q)r-\lfloor |i|/q\rfloor\big)\leq \sum_{p\in\ell}\mult_\ell(P^{(i,0)},p)$. It follows from the one-dimensional case of \Cref{lemma:SZ} that the univariate polynomial obtained
  by restricting $P^{(i,0)}$ to the line $\ell$ is the zero polynomial.

  Write $P=P_0+P_1+\dotsb+P_d$ where $P_k$ is the degree-$k$ homogeneous component of $P$ and $d\eqdef\nobreak \deg P$,
  and also write the line $\ell$ as $\ell=\{a+bt:t\in\Fq\}$, where $b=(b_1,\dotsc,b_{n-1},1)$.
  Note that, by the second part of \Cref{prop:hasse}, $P_d^{(i,0)}$ is the homogeneous part of $P^{(i,0)}$ of degree $d-\abs{i}$, and $\deg P^{(i,0)}\leq d-|i|$. Thus, $P_d^{(i,0)}(b)$ is the coefficient of $t^{d-\abs{i}}$ in the univariate polynomial $P^{(i,0)}(a+bt)$. Therefore, $P_d^{(i,0)}(b)=0$ for all $\abs{i}< r$ and for all~$b=(b_1,\dotsc,b_{n-1},1)$.

  Define the polynomial $Q\in\Fq[x_1,\dotsc,x_{n-1}]$ by
  $Q(x_1,\dotsc,x_{n-1})\eqdef P_d(x_1,\dotsc,x_{n-1},1)$. Note that $Q\neq 0$  since $P_d\neq 0$ and $P_d$ is homogeneous. For every $b\in \Fq^{n-1}$ and every $i\in\mathbb{Z}_{\geq 0}^{n-1}$ such that $\abs{i}<r$,
  we have \break$Q^{(i)}(b)=P^{(i,0)}(b,1)=0$. Hence, $\mult(Q,b)\geq r$ for all $b\in\Fq^{n-1}$, and so the generalized Schwartz--Zippel lemma implies that $\deg Q\geq rq$, which contradicts
  the definition of the space~$V$.
\end{proof}

The next task is to estimate the number of independent linear conditions that different lines
impose at the point~$p$. To that end, for any $p\in K$ and line $\ell\in L_p$, we set $W_{p,\ell}\subseteq \Fq[x_1,\dotsc,x_n]$ to be the subspace consisting of all polynomials vanishing to order $\bigl(r,(2-1/q)r\bigr)$ at $p$ along $\ell$.

\begin{lemma}\label{lemma:Codim}
  Let $p\in K$ be arbitrary.
Then the codimension of $W_p\eqdef \bigcap_{\ell\in L_p} W_{p,\ell}$ in $\Fq[x_1,x_2,\dotsc,x_n]$ satisfies 
\[
\codim W_p\leq \bigl((2-1/q)^n+m(n-1)(1-1/q)\bigr)\frac{r^n}{n!}+O_{n,q}(r^{n-1}),
\]
where $m=|L_p|$.
\end{lemma}

\begin{proof}
  Suppose $\ell=\{a+bt:t\in\Fq\}$ and $p=a+bt_0$ is a point on $\ell$.
  For $i\in \Znneg^{n-1}$ and $j\in \Znneg$, write $W_{p,\ell}(i,j)$ for the space of polynomials $P$ such that
  the univariate polynomial $Q=P^{(i,0)}(a+bt)$ satisfies $Q^{(j)}(t_0)=0$.
  Note that $W_{p,\ell}(i,j)$ is a subspace of codimension $1$.
  Since $\bigcap_{j<r'} W_{p,\ell}(i,j)$ consists of polynomials $P$ satisfying $\mult_{\ell}(P^{(i,0)},p)\geq r'$,
  it follows that
  \[
    W_{p,\ell}=\bigcap_{\substack{\abs{i}<r\\j<(2-1/q)r-\abs{i}/q}} W_{p,\ell}(i,j).
  \]

  For the purpose of proving the lemma, we may assume that $p$ is the origin and so $\ell=\{bt : t\in\Fq\}$.
  Given a polynomial $P$, write it as $P(x)=\sum_\alpha c_\alpha x^\alpha$. 
  Observe that, by the second part of \Cref{prop:hasse}, the $t^{j-|i|}$ coefficient of $P^{(i,0)}(bt)$ is a linear combination of the coefficients $c_\alpha$ whose $\alpha=\abs{i}+j$. Thus, the linear condition $P\in W_{p,\ell}(i,j)$ involves only such $c_\alpha$.

  We use this observation to separate the linear conditions into those affecting coefficients of degree $\abs{\alpha}<(2-1/q)r$
  from the rest. To that end we define
  \begin{equation}\label{eq:wpsep}
    W_{p,\ell}^{-}\eqdef \bigcap_{\substack{\abs{i}<r\\j<(2-1/q)r-\abs{i}/q\\\abs{i}+j<(2-1/q)r}} W_{p,\ell}(i,j)\qquad\text{and}\qquad
    W_{p,\ell}^{+}\eqdef \bigcap_{\substack{\abs{i}<r\\j<(2-1/q)r-\abs{i}/q\\\abs{i}+j\geq (2-1/q)r}} W_{p,\ell}(i,j).
  \end{equation}
  Set $W_p^{-}\eqdef \bigcap_{\ell\in L_p} W_{p,\ell}^{-}$ and $W_p^{+}\eqdef \bigcap_{\ell\in L_p} W_{p,\ell}^{+}$. Clearly,
  \begin{equation}\label{eq:wpadditive}
    \codim W_p=\codim W_p^{-}+\codim W_p^{+}.
  \end{equation}
    
  The $\codim W_p^{-}$ is easy to bound: the linear conditions in the definition of $W_p^{-}$ involve only
  the coefficients $c_{\alpha}$ of monomials of degree $\abs{\alpha}<(2-1/q)r$, and so we may bound $\codim W_p^{-}$
  by the number of such monomials, i.e.,
  \begin{equation}\label{eq:wpminus}
    \codim W_p^{-}\leq \binom{(2-1/q)r+n-1}{n}=\frac{\bigl((2-1/q)r\bigr)^n}{n!}+O_{n,q}(r^{n-1}).
  \end{equation}

  To estimate $\codim W_p^+$, we upper bound $\codim W_{p,\ell}^{+}$, for each $\ell\in L_p$, by the number of linear conditions appearing in \eqref{eq:wpsep}, i.e., 
  \begin{align}\label{eq:wpplusbound}
    \codim W_p^{+}&\leq m\cdot \abs[\big]{\bigl\{(i,j)\in \Znneg^{n-1}\times \Znneg :\abs{i}<r, j<(2-1/q)r-\abs{i}/q, \abs{i}+j\geq (2-1/q)r\bigr\}}.\\
  \intertext{Define the polytope}
    \polytope &\eqdef \bigl\{(i,j)\in \Rnneg^{n-1}\times \Rnneg :\abs{i}<1, j<(2-1/q)-\abs{i}/q,\, \abs{i}+j\geq 2-1/q\bigr\}.
  \end{align}
  Since the set on the right side of \eqref{eq:wpplusbound} is the set of lattice points in $r\cdot \polytope$, we may
  write the bound \eqref{eq:wpplusbound} as $\codim W_p^+\leq m\cdot r^n\vol (\polytope) + O_{n,q}(r^{n-1})$.

  The $\polytope$ can be expressed as a Boolean combination of three simpler polytopes,
  \begin{align*}
    \squarepolytope\,&\eqdef \bigl\{(i,j)\in\mathbb{R}_{\geq 0}^{n-1}\times\mathbb{R}_{\geq 0}:\abs{i}<1, 1-1/q<j\leq 2-2/q\bigr\},\\
    \simplexone&\eqdef \bigl\{(i,j)\in\mathbb{R}_{\geq 0}^{n-1}\times\mathbb{R}_{\geq 0}:\abs{i}< 1, j> 1-1/q, \abs{i}+j< 2-1/q\},\\
    \simplextwo&\eqdef \{(i,j)\in\mathbb{R}_{\geq 0}^{n-1}\times\mathbb{R}_{\geq 0}:\abs{i}< 1,2-2/q< j< (2-1/q)-\abs{i}/q \}.
  \end{align*}
  The $\simplexone$ and $\simplextwo$ are $n$-simplices, whereas $\squarepolytope$ is a cylinder over an $(n-1)$-simplex. We can depict
  these polytopes as follows.
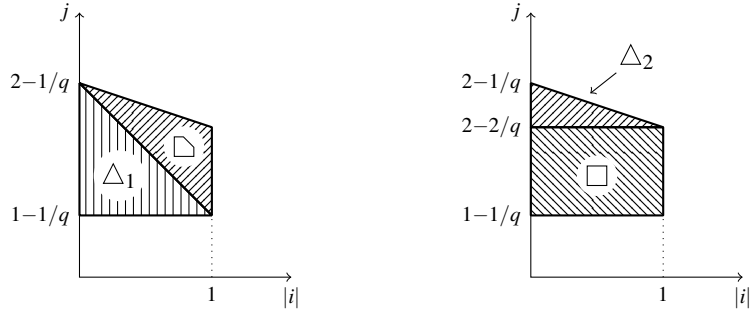
\begin{center}
\begin{tikzpicture}[polytope/.style={thick,line join=bevel},polytope label/.style={shape=circle,inner sep=1pt,fill=white,minimum size=1.4em}]
  \def\smidgen{0.3333}             % 1/q
  \def\axisheight{2.2}
  \def\axiswidth{1.6}
  \def\bottomy{0.2}
  \def\scalefactor{1.76}
  \def\drawaxes{% 
    \draw[<->] (0,\axisheight) -- (0,\bottomy)--(\axiswidth,\bottomy);
    \node[anchor=east] at (0,2-\smidgen) {$\scriptstyle 2-1/q$};
    \node[anchor=east] at (0,1-\smidgen) {$\scriptstyle 1-1/q$};
    \node[anchor=east] at (0,\axisheight) {$\scriptstyle j$};
    \node[anchor=north] at (1,\bottomy) {$\scriptstyle 1$};
    \node[anchor=north] at (\axiswidth,\bottomy) {$\scriptstyle \abs{i}$};
    \draw[dotted] (1,\bottomy) -- (1,1-\smidgen);
  }
  \begin{scope}[scale=\scalefactor]
    \drawaxes;
    \draw[polytope,pattern=north east lines] (0,2-\smidgen) -- (1,2-2*\smidgen) -- (1,1-\smidgen)  -- cycle;
    \draw[polytope,pattern=vertical lines] (0,2-\smidgen) -- (1,1-\smidgen) -- (0,1-\smidgen) -- cycle;
    \node[polytope label] at (0.79,1.52-\smidgen) {$\polytope$};
    \node[polytope label] at (0.3,1.3-\smidgen) {$\simplexone$};
  \end{scope}
  \begin{scope}[xshift=6cm,scale=\scalefactor]
    \drawaxes;
    \node[anchor=east] at (0,2-2*\smidgen) {$\scriptstyle 2-2/q$};    
    \draw[polytope,pattern=north east lines] (0,2-\smidgen) -- (1,2-2*\smidgen) -- (0,2-2*\smidgen)  -- cycle;
    \draw[polytope,pattern=north west lines] (0,2-2*\smidgen) -- (1,2-2*\smidgen) -- (1,1-\smidgen) -- (0,1-\smidgen) -- cycle;
    \node[polytope label] (simplextwo) at (0.8,2.2-\smidgen) {$\simplextwo$};
    \node[polytope label] at (0.5,1.3-\smidgen) {$\squarepolytope$};
    \draw[->] (simplextwo) -- (0.45,1.59);
  \end{scope}
\end{tikzpicture}
\end{center}

Since $\polytope+\simplexone=\squarepolytope+\simplextwo$, with pluses denoting disjoint unions, it follows that
\begin{align*}
  \vol(\polytope)&=\vol(\squarepolytope)-\vol(\simplexone)+\vol(\simplextwo)\\
  &=\tfrac{1-1/q}{(n-1)!}-\tfrac{1}{n!}+\tfrac{1}{n!q}\\
  &=\tfrac{(n-1)(1-1/q)}{n!}.
\end{align*}

We thus obtain
\begin{equation}\label{eq:wpplusfinal}
  \codim W_p^{+}\leq m\tfrac{(n-1)(1-1/q)}{n!}r^n+O_{n,q}(r^{n-1}).
\end{equation}
Combining \eqref{eq:wpadditive}, \eqref{eq:wpminus} and \eqref{eq:wpplusfinal} we obtain the desired result.
\end{proof}

To get a lower bound on $\abs{K}$ it remains to estimate $\dim V$. On one hand, it is equal to
\begin{align}
\dim V=\abs{A}=&\sum_{\alpha_n=0}^{(1-1/q)rq-1}\binom{rq+n-2}{n-1}+\sum_{\alpha_n=(1-1/q)rq}^{(2-1/q)rq-1}\binom{(2-1/q)rq-\alpha_n+n-2}{n-1} \notag\\
=&(1-1/q)rq\binom{rq+n-2}{n-1}+\binom{rq+n-1}{n} \notag\\
=&\frac{(1-1/q)r^nq^n}{(n-1)!}+\frac{r^nq^n}{n!}-O_{n,q}(r^{n-1}).\label{eq:vasympt}
\end{align}
On the other hand, from \Cref{lemma:ZeroPoly} we know that $\bigcap_{p\in K}W_p\cap V=\{0\}$, and hence
\begin{align*}
  \dim V&\leq \codim \bigcap_{p\in K}W_p\leq \sum_{p\in K}\codim W_p\\&\leq \sum_{p\in K}\bigl((2-1/q)^n+\abs{L_p}(n-1)(1-1/q)\bigr)\frac{r^n}{n!}+O_{n,q}(r^{n-1})\qquad\text{ by \Cref{lemma:Codim}}.
\end{align*}
The sum above can be computed by noting that $\sum \abs{L_p}=q\abs{L}=q^n$. We then let $r\to\infty$ and compare the resulting upper bound on $\dim V$ with the asymptotics in \eqref{eq:vasympt}
to obtain
\[
 n(1-1/q)q^n+q^n\leq \abs{K}(2-1/q)^n+q^n(n-1)(1-1/q).
\]
By rearranging the inequality, we see that this is equivalent to
\[
  \abs{K}\geq \frac{q^n}{(2-1/q)^{n-1}}.
\]

\section{Lower-order terms}\label{sec:problem}
In dimension $2$, one can use
the simple fact that two distinct lines intersect at most
once to derive the lower bound of $\abs{K}\geq \frac{q(q+1)}{2}$
for every $2$-dimensional Kakeya set. Though this bound is sharp
for even $q$, the sharp bound for odd $q$ is $\abs{K}\geq \frac{q(q+1)}{2}+\frac{q-1}{2}$,
as shown by Blokhuis and Mazzocca \cite{blokhuis_mazzocca}.

Recall that we defined a set to be \emph{almost Kakeya} if it contains a line in every non-horizontal direction.
The paper \cite{saraf_sudan} presents two constructions of higher-dimensional Kakeya sets, one construction 
for each possible parity of~$q$. The construction for odd $q$ (due to Dvir) relies on an auxiliary construction
of almost Kakeya sets, whereas the construction for even $q$ (due to Kopparty, Saraf, and Sudan)
is direct.

For odd $q$, the almost Kakeya sets in \cite{saraf_sudan} are of size $q\cdot (\frac{q+1}{2})^{n-1}$.
For large $q$, this quantity is $2^{-n+1}q^n\bigl(1+\frac{n-1}{q}+O_n(q^{-2})\bigr)$ whereas
the lower bound in \Cref{thm:high} is $2^{-n+1}q^n\bigl(1+\frac{n-1}{2q}+O_n(q^{-2})\bigr)$.
It would be interesting to close the gap.

To turn an almost Kakeya set into a genuine Kakeya set, one must
take care of the horizontal directions. In \cite{saraf_sudan} this was
achieved by adding a horizontal hyperplane. This can be done
more efficiently by adding a lower-dimensional Kakeya set instead.
\begin{proposition}
  There is a Kakeya set in $\Fq^n$ of size $2^{-n+1}q^n\bigl(1+\tfrac{n+1-2^{-n+2}}{q}+O_n(q^{-2})\bigr)$ if $q$ is odd.
\end{proposition}
\begin{proof}
  By induction on $n$, with the base case $n=1$ being trivial. Let $K_n'$ be an
  almost Kakeya set in $\Fq^n$ of size $q\cdot (\frac{q+1}{2})^{n-1}$.
  Let $K_{n-1}$ be the inductively-constructed $(n-1)$-dimensional Kakeya set
  of size $2^{-n+2}q^{n-1}+O_n(q^{n-2})$. We think of $K_{n-1}$ as lying inside a horizontal
  hyperplane in $\Fq^n$. Then the set $K_n\eqdef K_n'\cup (K_{n-1}+x)$ is a Kakeya set
  for any choice of $x\in \Fq^n$. For a random choice of $x$, 
  \[
  \E\bigl[\abs{K_n}\bigr]=\abs{K_n'}+\abs{K_{n-1}}\left(1-\frac{\abs{K_n'}}{q^n}\right)=2^{-n+1}q^n\bigl(1+\tfrac{n+1-2^{-n+2}}{q}+O_n(q^{-2})\bigr).\qedhere
  \]
\end{proof}

For even values of $q$, the Kakeya sets constructed in \cite{saraf_sudan} are of size 
$2^{-n+1}q^n+(1-2^{-n+1})q^{n-1}$.

%%% AUTHOR: optional acknowledgments here
\section*{Acknowledgments} %%  you may comment this out if no Ackno
Throughout this work we benefited from discussions with N\'ora Frankl.
We are thankful to the anonymous referee for very astute and useful comments.

%%% AUTHOR:
%%% Bibliography goes here. Note that the arXiv cannot process bibtex
%%% or biber bibliographies.  Example of acceptable bibliograpy format:
\bibliographystyle{amsplain}

%% AUTHOR: You can generate such a bibliography from a .bib file by 
%% running pdflatex/bibtex/pdflatex/pdflatex and then pasting the .bbl file
%% between \begin{thebibliography} and \end{bibliography}

%%% AUTHOR: Include a short description of each author following the
%%% structure below. Use the same short tags used previously.  
%%% Use \imageat{} and \imagedot{} instead of "@" and "." in
%%% email addresses-this replaces the symbols with graphics to avoid 
%%% e-mail address harvesting from the .pdf file
\begin{dajauthors}
\begin{authorinfo}[pgom]
  Boris Bukh\\
  Department of Mathematical Sciences, Carnegie Mellon University\\
  bbukh\imageat{}math\imagedot{}cmu\imagedot{}edu \\
  \url{http://www.borisbukh.org/}
\end{authorinfo}
\begin{authorinfo}[joha]
  Ting-Wei Chao\\
  Department of Mathematical Sciences, Carnegie Mellon University\\
  tchao2\imageat{}andrew\imagedot{}cmu\imagedot{}edu
\end{authorinfo}
\end{dajauthors}

\end{document}